\documentclass[a4paper,12pt]{article}
\usepackage{amssymb,amsmath,mathrsfs,euscript,amsopn}

\numberwithin{equation}{section}

\newtheorem{lemma}{Lemma}[section]

\newtheorem{theorem}{Theorem}[section]
\newtheorem{proposition}{Proposition}[section]
\newtheorem{corollary}{Corollary}[section]

\newtheorem{remark}{Remark}[section]

\newenvironment{proof}{\smallskip\noindent{\bf Proof.}\rm}{\hspace*{\fill} $\Box$\medskip}

\newenvironment{proofth}{\smallskip\noindent{\bf Proof of Theorem~\ref{th1}.}\rm}{\hspace*{\fill} $\Box$\medskip}

\renewcommand{\Re}{\operatorname{Re}}

\title{Bari--Markus property for Dirac operators}

\author{
Ya.~V.~Mykytyuk, D.~V.~Puyda\thanks{\emph{Email addresses:} yamykytyuk@yahoo.com (Ya.~V.~Mykytyuk), dpuyda@gmail.com (D.~V.~Puyda)}\\
\emph{Ivan Franko National University of Lviv}\\
\emph{1 Universytetska str., Lviv, 79000, Ukraine}
}

\date{}

\begin{document}

\maketitle

\begin{abstract}
We prove the Bari--Markus property for spectral projectors of non-self-adjoint Dirac operators on $(0,1)$ with square-integrable matrix-valued potentials and some separated boundary conditions.
\end{abstract}

\section{Introduction and main results}

In the Hilbert space $\mathbb H:=L_2((0,1),\mathbb C^{2r})$, we study the non-self-adjoint Dirac operator
$$
T_Q:=J\frac{\mathrm d}{\mathrm dx}+Q
$$
on the domain
$$
D(T_Q):=\left\{(y_1,y_2)^\top\mid y_1,y_2\in W_2^1((0,1),\mathbb C^r), \quad  y_1(0)=y_2(0),\ y_1(1)=y_2(1)\right\}.
$$
Here,
$$
J:=\frac1{\mathrm i}\begin{pmatrix}I&0\\0&-I\end{pmatrix},\qquad
Q:=\begin{pmatrix}0&q_1\\q_2&0\end{pmatrix},
$$
$I:=I_r$ is the $r\times r$ identity matrix, $q_1,q_2\in L_2((0,1),\mathcal M_r)$, $\mathcal M_r$ is the set of $r\times r$ matrices with complex entries and $W_2^1((0,1),\mathbb C^r)$ is the Sobolev space of $\mathbb C^r$-valued functions.
All functions $Q$ as above form the set
$$
\mathfrak Q_2:=\{Q\in L_2((0,1),\mathcal M_{2r})\mid JQ(x)=-Q(x)J\,\,\text{a.e. on}\,\,(0,1)\}
$$
and will be called \emph{potentials} of the operators $T_Q$.

The spectrum $\sigma(T_Q)$ of the operator $T_Q$ consists of countably many isolated eigenvalues of finite algebraic multiplicities. We denote by $\lambda_j:=\lambda_j(Q)$, $j\in\mathbb Z$, the pairwise distinct eigenvalues of the operator $T_Q$ arranged by non-decreasing of their real -- and then, if equal, imaginary -- parts. For definiteness, we also assume that $\Re\lambda_0\le0<\Re\lambda_1$.
As can be proved using the standard technique based on Rouche's theorem, the numbers $\lambda_j$, $j\in\mathbb Z$, satisfy the condition
\begin{equation}\label{lambdaCond}
\sup_{n\in\mathbb Z}\sum_{\lambda_j\in\Delta_n}1<\infty
\end{equation}
and the asymptotics
\begin{equation}\label{lambdaAsymp}
\sum_{n\in\mathbb Z}\sum_{\lambda_j\in\Delta_n}|\lambda_j-\pi n|^2<\infty,
\end{equation}
where $\Delta_n:=\{\lambda\in\mathbb C\mid\pi n-\pi/2<\Re\lambda\le\pi n+\pi/2\}$, $n\in\mathbb Z$.
We then denote by $P_{\lambda_j}$ the spectral projector of the operator $T_Q$ corresponding to the eigenvalue $\lambda_j$ (see \cite[Chap.3]{kato}). We write
$$
\mathcal P_n:=\sum_{\lambda_j\in\Delta_n} P_{\lambda_j},\qquad n\in\mathbb Z,
$$
for the spectral projector of $T_Q$ corresponding to the strip $\Delta_n$.

In particular, in the free case $Q=0$ one has $\sigma(T_0)=\{\pi n\}_{n\in\mathbb Z}$. We then write $\mathcal P_n^0$ for the spectral projector of the free operator $T_0$ corresponding to the strip $\Delta_n$, $n\in\mathbb Z$.

The main result of this paper is the following theorem:
\begin{theorem}\label{th1}
For every $Q\in\mathfrak Q_2$, it holds
\begin{equation}\label{bm}
\sum_{n\in\mathbb Z}\left\|\mathcal P_n-\mathcal P_n^0\right\|^2<\infty \,.
\end{equation}
\end{theorem}

\noindent Relation (\ref{bm}) is called the \emph{Bari--Markus property} of spectral projectors of the operator $T_Q$.

In the scalar case $r=1$, the Bari--Markus property for the operator $T_Q$, as well as for the operators with periodic and anti-periodic boundary conditions, was established in \cite{bm1} to prove the unconditional convergence of spectral decompositions for such operators. Therein, P.~Djakov and B.~Mityagin used a technique based on Fourier representations of Dirac operators. This technique was further developed to prove the similar property for Dirac operators with regular boundary conditions in \cite{bm2}. For Hill operators with singular potentials, the Bari--Markus property was established in \cite{bmhill}.

A different and simpler technique based on some convenient representation of resol\-vents of the operators under consideration was used in \cite{sturm} to establish the Bari--Markus property for Sturm--Liouville operators with matrix-valued potentials (see \cite[Lemma~2.12]{sturm}). Therein, this result was used to solve the inverse spectral problem for such operators. For the same purpose, the Bari--Markus property was established for self-adjoint Dirac operators with square-integrable matrix-valued potentials in \cite{dirac1}.

In the present paper, we use the technique suggested in \cite{sturm} to establish the Bari--Markus property for non-self-adjoint Dirac operators with square-integrable matrix-valued potentials. This result  can be used to study the inverse spectral problems for non-self-adjoint Dirac operators on a finite intervals.

The paper is organized as follows. In the reminder of this sections, we introduce some notations that are used in this paper. In Sects.~\ref{sect2} and \ref{sect3}, we provide some preliminary results and prove Theorem~\ref{th1}, respectively.

\emph{Notations.} Throughout this paper, we identify $\mathcal M_r$ with the Banach algebra of linear operators in $\mathbb C^r$ endowed with the standard norm. If there is no ambiguity, we write simply $\|\cdot\|$ for norms of operators and matrices.

We denote by $L_2((a,b),\mathcal M_r)$ the Banach space of all strongly measurable functions $f:(a,b)\to\mathcal M_r$ for which the norm
$$
\|f\|_{L_2}:=\left(\int_a^b \|f(t)\|^2\mathrm dt\right)^{1/2}
$$
is finite. We denote by $G_2(\mathcal M_r)$ the set of all measurable functions $K:[0,1]^2\to \mathcal M_r$ such that for all $x,t\in[0,1]$, the functions $K(x,\cdot)$ and $K(\cdot,t)$ belong to $L_2((0,1),\mathcal M_r)$ and, moreover, the mappings
$[0,1]\ni x\mapsto K(x,\cdot)\in L_2((0,1),\mathcal M_r)$ and
$[0,1]\ni t\mapsto K(\cdot,t)\in L_2((0,1),\mathcal M_r)$
are continuous. We denote by $G_2^+(\mathcal M_r)$ the set of all functions $K\in G_2(\mathcal M_r)$ such that $K(x,t)=0$ a.e. in the triangle $\Omega_-:=\{(x,t)\mid0<x<t<1\}$.
The superscript $\top$ designates the transposition of vectors and matrices.

\section{Preliminary results}\label{sect2}

In this section, we obtain some preliminary results and introduce some auxiliary objects that will be used in this paper.

For an arbitrary potential $Q\in\mathfrak Q_2$ and $\lambda\in\mathbb C$, we denote by $Y_Q(\cdot,\lambda)\in W_2^1((0,1),\mathcal M_{2r})$ the $2r\times2r$ matrix-valued solution of the Cauchy problem
\begin{equation}\label{YCauchyProbl}
J\frac{\mathrm d}{\mathrm dx} Y+QY=\lambda Y,\qquad Y(0,\lambda)=I_{2r}.
\end{equation}
We set
$\varphi_Q(\cdot,\lambda):=Y_Q(\cdot,\lambda)Ja^*$ and $\psi_Q(\cdot,\lambda):=Y_Q(\cdot,\lambda)a^*$, where $a:=\frac1{\sqrt2}\begin{pmatrix}I,&-I\end{pmatrix}$,
so that $\varphi_Q(\cdot,\lambda)$ and $\psi_Q(\cdot,\lambda)$ are the $2r\times r$ matrix-valued solutions of the Cauchy problems
\begin{equation}\label{phiCauchyProbl}
J\frac{\mathrm d}{\mathrm dx}\varphi+Q\varphi=\lambda\varphi,\qquad \varphi(0,\lambda)=Ja^*,
\end{equation}
and
$$
J\frac{\mathrm d}{\mathrm dx}\psi+Q\psi=\lambda\psi,\qquad \psi(0,\lambda)=a^*,
$$
respectively. For an arbitrary $\lambda\in\mathbb C$, we introduce the operator $\Phi_Q(\lambda):\mathbb C^r\to\mathbb H$ by the formula
$$
[\Phi_Q(\lambda)c](x):=\varphi_Q(x,\lambda)c,\qquad x\in[0,1].
$$

We set $s_Q(\lambda):=a\varphi_Q(1,\lambda)$ and $c_Q(\lambda):=a\psi_Q(1,\lambda)$,
$\lambda\in\mathbb C$.
The function
$$
m_Q(\lambda):=-s_Q(\lambda)^{-1}c_Q(\lambda)
$$
will be called the \emph{Weyl--Titchmarsh function} of the operator $T_Q$.
Note that in the free case $Q=0$ one has $s_0(\lambda)=(\sin\lambda)I$, $c_0(\lambda)=(\cos\lambda)I$ and $m_0(\lambda)=-(\cot\lambda)I$.

The following proposition is a straightforward analogue of Lemma~2.1 in \cite{dirac1}:

\begin{proposition}\label{scProp}
For an arbitrary potential $Q\in\mathfrak Q_2$ it holds:
\begin{itemize}
\item[(i)]there exists a unique function $K_Q\in G_2^+(\mathcal M_{2r})$ such that for every $x\in[0,1]$ and $\lambda\in\mathbb C$,
    $$
    \varphi_Q(x,\lambda)=\varphi_0(x,\lambda)+\int_0^x K_Q(x,s)\varphi_0(s,\lambda)\ \mathrm d s,
    $$
    where $\varphi_0(\cdot,\lambda)$ is a solution of (\ref{phiCauchyProbl}) in the free case $Q=0$;
\item[(ii)]there exist unique functions $f_1:=f_{Q,1}$ and $f_2:=f_{Q,2}$ from $L_2((-1,1),\mathcal M_r)$ such that for every $\lambda\in\mathbb C$,
    \begin{equation}\label{scRepr}
    s_Q(\lambda)=(\sin\lambda)I+\frac1{\sqrt2}\int_{-1}^1 \mathrm e^{\mathrm i\lambda s}f_1(s)\ \mathrm d s,\quad
    c_Q(\lambda)=(\cos\lambda)I+\frac1{\sqrt2}\int_{-1}^1 \mathrm e^{\mathrm i\lambda s}f_2(s)\ \mathrm d s.
    \end{equation}
\end{itemize}
\end{proposition}

\noindent In particular, Proposition~\ref{scProp} implies the following corollary:

\begin{corollary}\label{scCor}
For an arbitrary $Q\in\mathfrak Q_2$ and $\lambda\in\mathbb C$,
\begin{equation}\label{phiTrOp}
\Phi_Q(\lambda)=(\mathcal I+\mathcal K_Q)\Phi_0(\lambda),
\end{equation}
where $\mathcal K_Q$ is the integral operator with kernel $K_Q$ and $\mathcal I$ is the identity operator in $\mathbb H$.
\end{corollary}

\noindent Using the first formula in (\ref{scRepr}) and repeating the proof of Theorem~3 in \cite{trushzeros}, one can also derive the following:
\begin{corollary}\label{zerosCor}
The set of zeros of the entire function $\widetilde s_Q(\lambda):=\det s_Q(\lambda)$ can be indexed (counting multiplicities) by numbers $n\in\mathbb Z$ so that the corresponding sequence $(\xi_n)_{n\in\mathbb Z}$ has the asymptotics
$$
\xi_{kr+j}=\pi k+\omega_{j,k},\qquad k\in\mathbb Z,\quad j=0,\ldots,r-1,
$$
where the sequences $(\omega_{j,k})_{k\in\mathbb Z}$ belong to $\ell_2(\mathbb Z)$.
\end{corollary}

Now let $\rho(T_Q)$ denote the resolvent set of the operator $T_Q$.

\begin{lemma}\label{dirProp}
For an arbitrary $Q\in\mathfrak Q_2$ it holds $\rho(T_Q)=\{\lambda\in\mathbb C\mid\ker s_Q(\lambda)=\{0\}\}$ and for each $\lambda\in\rho(T_Q)$,
\begin{equation}\label{resolv}
(T_Q-\lambda\mathcal I)^{-1}=\Phi_Q(\lambda)m_Q(\lambda)\Phi_{Q^*}(\overline\lambda)^*
+\mathcal T_Q(\lambda),
\end{equation}
where $\mathcal T_Q$ is an entire operator-valued function. The spectrum of the operator $T_Q$ consists of countably many isolated eigenvalues of finite algebraic multiplicities.
\end{lemma}

\begin{proof}
A direct verification shows that
$$
\frac{\mathrm d}{\mathrm dx}\left(JY_{Q^*}(x,\overline\lambda)^*JY_Q(x,\lambda)\right)=0.
$$
Therefore, taking into account (\ref{YCauchyProbl}), we find that $-JY_{Q^*}(x,\overline\lambda)^*JY_Q(x,\lambda)=I_{2r}$ for every $x\in[0,1]$ and thus
$$
Y_Q(x,\lambda)JY_{Q^*}(x,\overline\lambda)^*=J,\qquad x\in[0,1].
$$
Since $J=Ja^*a+a^*aJ$, the latter can be rewritten as
\begin{equation}\label{ResolvAuxEq1}
\varphi_Q(x,\lambda)\psi_{Q^*}(x,\overline\lambda)^*-\psi_Q(x,\lambda)\varphi_{Q^*}(x,\overline\lambda)^*
=J,\qquad x\in[0,1].
\end{equation}

Using (\ref{ResolvAuxEq1}), one can verify that for an arbitrary $f\in\mathbb H$ and $\lambda\in\mathbb C$, the function
$$
g(x,\lambda)
=[\mathcal T_Q(\lambda)f](x):=\psi_Q(x,\lambda)\int_0^x\varphi_{Q^*}(t,\overline\lambda)^*f(t)\ \mathrm dt+
\varphi_Q(x,\lambda)\int_x^1\psi_{Q^*}(t,\overline\lambda)^*f(t)\ \mathrm dt
$$
solves the Cauchy problem
\begin{equation}\label{ResolvAuxEq2}
Jy'+Qy=\lambda y+f,\qquad y_1(0)=y_2(0).
\end{equation}
Since for every $c\in\mathbb C^r$, the function $h(\cdot,\lambda):=\varphi_Q(\cdot,\lambda)c$ solves (\ref{ResolvAuxEq2}) with $f=0$, it then follows that a generic solution of (\ref{ResolvAuxEq2}) takes the form $y=\varphi_Q(\cdot,\lambda)c+\mathcal T_Q(\lambda)f$, $c\in\mathbb C^r$.
If $\lambda\in\mathbb C$ is such that the $r\times r$ matrix $s_Q(\lambda):=a\varphi_Q(1,\lambda)$ is non-singular, then the choice
$$
c=-s_Q(\lambda)^{-1}c_Q(\lambda)\int_0^1\varphi_{Q^*}(t,\overline\lambda)^*f(t)\ \mathrm dt
$$
implies that $ay(1)=0$, i.e. $y_1(1)=y_2(1)$. Therefore, every $\lambda\in\mathbb C$ such that $\ker s_Q(\lambda)=\{0\}$ is a resolvent point of the operator $T_Q$ and for such $\lambda$ it holds
$$
(T_Q-\lambda\mathcal I)^{-1}=\Phi_Q(\lambda)m_Q(\lambda)\Phi_{Q^*}(\overline\lambda)^*+\mathcal T_Q(\lambda).
$$

To complete the proof, it remains to observe that the function $y=\varphi_Q(\cdot,\lambda)c$ is a non-zero solution of the problem
$$Jy'+Qy=\lambda y, \qquad  y_1(0)=y_2(0), \quad y_1(1)=y_2(1)$$
if and only if $c\in\ker s_Q(\lambda)\setminus\{0\}$.
Since the values of the resolvent of the operator $T_Q$ are compact operators, it follows that all spectral projectors $P_{\lambda_j}$, $j\in\mathbb Z$, are finite dimensional. In particular, it then follows (see, e.g., \cite[Theorem~2.2]{kreinnsa}) that all eigenvalues of the operator $T_Q$ are of finite algebraic multiplicities.
\end{proof}

From Lemma~\ref{dirProp} we obtain that eigenvalues of the operator $T_Q$ are zeros of the entire function $\widetilde s_Q(\lambda):=\det s_Q(\lambda)$. In view of Corollary~\ref{zerosCor} we then arrive at the following:

\begin{corollary}\label{asCor}
For an arbitrary potential $Q\in\mathfrak Q_2$, eigenvalues of the operator $T_Q$ satisfy the condition (\ref{lambdaCond}) and the asymptotics (\ref{lambdaAsymp}).
\end{corollary}

Now we can introduce the spectral projectors of the operator $T_Q$ as explained in the previous section. Formulas (\ref{phiTrOp}) and (\ref{resolv}) will serve as an efficient tool to prove Theorem~\ref{th1}.

\section{Proof of Theorem~\ref{th1}}\label{sect3}

Now we are ready to prove Theorem~\ref{th1}. We start with the following auxiliary lemma:

\begin{lemma}\label{auxLemma}
For an arbitrary $\lambda\in\mathbb C$, let the operator $A(\lambda):L_2((-1,1),\mathcal M_r)\to\mathcal M_r$ act by the formula
$$
A(\lambda)f:=\frac1{\sqrt2}\int_{-1}^1 \mathrm e^{\mathrm i\lambda t}f(t)\,\mathrm dt.
$$
Then for an arbitrary $f\in L_2((-1,1),\mathcal M_r)$ and $\lambda\in\mathbb T_0:=\{\lambda\in\mathbb C\mid|\lambda|=1\}$,
\begin{equation}\label{auxLemmaEq}
\sum_{n\in\mathbb Z}\|A(\pi n+\lambda)f\|^2\le 9r\|f\|_{L_2}^2.
\end{equation}
\end{lemma}

\begin{proof} Let $f\in L_2((-1,1),\mathcal M_r)$, $\lambda\in\mathbb T_0$ and $\|S\|_2$ denote the Hilbert--Schmidt norm of a matrix $S\in \mathcal M_r$. Since $\left\{\frac1{\sqrt2}\mathrm e^{\mathrm i\pi n t}\right\}_{n\in\mathbb Z}$ is an orthonormal basis in $L_2(-1,1)$, it follows that
$$
\sum_{n\in\mathbb Z}\|A(\pi n)f\|^2\le
\sum_{n\in\mathbb Z}\|A(\pi n)f\|_2^2=
\int^1_{-1}\|f(x)\|_2^2 \,\mathrm dx \le r \int^1_{-1}\|f(x)\|^2 \,\mathrm dx.
$$
Taking into account that $A(\pi n+\lambda)f=A(\pi n)f_1$ with $f_1(t):=\mathrm e^{\mathrm i\lambda t}f(t)$ and that $\|f_1\|_{L_2}<3\|f\|_{L_2}$, we then arrive at (\ref{auxLemmaEq}).
\end{proof}

\begin{remark}
In the notations of the above lemma, formulas (\ref{scRepr}) can be rewritten as
\begin{equation}\label{scRepr1}
s_Q(\lambda)=(\sin\lambda)I+A(\lambda)f_1,\qquad c_Q(\lambda)=(\cos\lambda)I+A(\lambda)f_2.
\end{equation}
\end{remark}

Now we are ready to prove Theorem~\ref{th1}:

\begin{proofth}
Recalling formula~(\ref{resolv}) and the asymptotics (\ref{lambdaAsymp}) of eigenvalues of the operator $T_Q$, we obtain that there exists $N\in\mathbb N$ such that for every $n\in\mathbb Z$ with $|n|>N$,
$$
\mathcal P_n:=-\frac1{2\pi\mathrm i}\oint_{\mathbb T_n}\Phi_Q(\lambda)m_Q(\lambda)\Phi_{Q^*}(\overline\lambda)^*\mathrm d\lambda,\qquad
\mathcal P_n^0:=-\frac1{2\pi\mathrm i}\oint_{\mathbb T_n}\Phi_0(\lambda)m_0(\lambda)\Phi_0(\overline\lambda)^*\mathrm d\lambda,
$$
where $\mathbb T_n:=\{\lambda\in\mathbb C\mid|\lambda-\pi n|=1\}$. Therefore, for each $n\in\mathbb Z$ such that $|n|>N$,
$$
\|\mathcal P_n-\mathcal P_n^0\|=\left\|-\frac1{2\pi\mathrm i}\oint_{\mathbb T_n}\left(
\Phi_Q(\lambda)m_Q(\lambda)\Phi_{Q^*}(\overline\lambda)^*
-\Phi_0(\lambda)m_0(\lambda)\Phi_0(\overline\lambda)^*\right)\mathrm d\lambda\right\|
\le\|\alpha_n\|+\|\beta_n\|,
$$
where
\begin{equation}\label{alphaNdef}
\alpha_n:=-\frac1{2\pi\mathrm i}\oint_{\mathbb T_n}\Phi_Q(\lambda)
(m_Q(\lambda)-m_0(\lambda))\Phi_{Q^*}(\overline\lambda)^*\mathrm d\lambda
\end{equation}
and
\begin{equation}\nonumber
\beta_n:=-\frac1{2\pi\mathrm i}\oint_{\mathbb T_n}\left(
\Phi_Q(\lambda)m_0(\lambda)\Phi_{Q^*}(\overline\lambda)^*
-\Phi_0(\lambda)m_0(\lambda)\Phi_0(\overline\lambda)^*
\right)\mathrm d\lambda.
\end{equation}
The theorem will be proved if we show that $\sum_{|n|>N}\|\alpha_n\|^2<\infty$ and $\sum_{|n|>N}\|\beta_n\|^2<\infty$.

Let us prove the claim for $(\alpha_n)$ first. Taking into account (\ref{scRepr1}), observe that
\begin{equation}\label{auxEq2}
m_Q(\lambda)-m_0(\lambda)=s_Q(\lambda)^{-1}\left[(\cot\lambda)A(\lambda)f_1-A(\lambda)f_2\right],
\end{equation}
where $A(\lambda)$ is from Lemma~\ref{auxLemma}. Note that by virtue of the Riemann--Lebesgue lemma, without loss of generality we may assume that
$$
\sup_{|n|>N}\sup_{\lambda\in\mathbb T_n}\|A(\lambda)f_1\|\le\frac14.
$$
Since for every $\lambda\in \mathbb T_n$ one has $|\sin\lambda|\ge1/2$, in view of the first formula in (\ref{scRepr1}) it then holds
$$
\|s_Q(\lambda)^{-1}\|\le|\sin\lambda|^{-1}(1-|\sin\lambda|^{-1}\|A(\lambda)f_1\|)^{-1}\le4,\qquad
\lambda\in\mathbb T_n,\quad |n|>N.
$$
Since $|\cot\lambda|\le\sqrt3$ as $\lambda\in \mathbb T_n$, from (\ref{auxEq2}) we then obtain that
\begin{equation}\label{auxEq1}
\|m_Q(\lambda)-m_0(\lambda)\|^2\le64(\|A(\lambda)f_1\|^2+\|A(\lambda)f_2\|^2),\qquad
\lambda\in\mathbb T_n,\quad |n|>N.
\end{equation}
Next, taking into account (\ref{phiTrOp}), observe that for an arbitrary $Q\in\mathfrak Q_2$ and $\lambda\in \mathbb T_n$ it holds
\begin{equation}\label{auxEq4}
\|\Phi_Q(\lambda)\|\le\|\mathcal I+\mathcal K_Q\|\|\Phi_0(\lambda)\|\le2\|\mathcal I+\mathcal K_Q\|.
\end{equation}
By virtue of the Cauchy--Bunyakovsky inequality we then obtain from (\ref{alphaNdef}), (\ref{auxEq1}) and (\ref{auxEq4}) that for every $n\in\mathbb Z$ such that $|n|>N$,
\begin{align*}
\|\alpha_n\|^2\le
C\int_0^{2\pi}\left(
\|A(\pi n+\mathrm e^{\mathrm i t})f_1\|^2+\|A(\pi n+\mathrm e^{\mathrm i t})f_2\|^2
\right)\mathrm d t
\end{align*}
with some $C>0$. In view of Lemma~\ref{auxLemma} we then obtain that $\sum_{|n|>N}\|\alpha_n\|^2<\infty$.

It thus only remains to prove that $\sum_{|n|>N}\|\beta_n\|^2<\infty$. For this purpose, take into account (\ref{phiTrOp}) and observe that
\begin{align*}
&\Phi_Q(\lambda)m_0(\lambda)\Phi_{Q^*}(\overline\lambda)^*
-\Phi_0(\lambda)m_0(\lambda)\Phi_0(\overline\lambda)^*
=
\\
&\qquad\mathcal K_Q \Phi_0(\lambda)m_0(\lambda)\Phi_0(\overline\lambda)^*
+\Phi_0(\lambda)m_0(\lambda)\Phi_0(\overline\lambda)^*\mathcal K_{Q^*}^* +\mathcal K_Q \Phi_0(\lambda)m_0(\lambda)\Phi_0(\overline\lambda)^*\mathcal K_{Q^*}^*.
\end{align*}
Therefore,
$
\beta_n=\mathcal K_Q \mathcal P_n^0+[{\mathcal K_{Q^*}}\mathcal P_n^0]^* +\mathcal K_Q \mathcal P_n^0 \mathcal K_{Q^*}^*
$
and thus the claim will be proved if we show that for an arbitrary $Q\in\mathfrak Q_2$,
\begin{equation}\label{auxEq3}
\sum_{|n|>N}\|\mathcal K_Q \mathcal P_n^0\|^2<\infty.
\end{equation}
To this end, note that the operator $\mathcal K_Q$ belongs to the Hilbert--Schmidt class $\mathcal B_2$ and that the sequence $(\mathcal P_n^0)_{n\in\mathbb Z}$ consists of pairwise orthogonal projectors. Therefore, it holds
$$
\sum_{n\in\mathbb Z}\|\mathcal K_Q \mathcal P_n^0\|^2
\le\sum_{n\in\mathbb Z}\|\mathcal K_Q \mathcal P_n^0\|_{\mathcal B_2}^2
\le\|\mathcal K_Q\|_{\mathcal B_2}^2.
$$
Hence (\ref{auxEq3}) follows and the proof is complete.
\end{proofth}


\begin{thebibliography}{9}
\bibitem{bm1}
P.~Djakov and B.~Mityagin, {\it Bari–-Markus property for Riesz projections of 1D periodic Dirac operators}, Math. Nachr. {\bf 283} (2010), 443–-462.

\bibitem{bmhill}
P.~Djakov and B.~Mityagin, {\it Bari--Markus property for Riesz projections of Hill operators with singular potentials}, Contemporary Math. {\bf 481} (2009), 59–-80.

\bibitem{bm2}
P.~Djakov and B.~Mityagin, {\it Unconditional convergence of spectral decompositions of 1D Dirac operators with regular boundary conditions}, Indiana University Math. Journal {\bf 61} (2012), 359--398.

\bibitem{kreinnsa}
I. C. Gohberg and M. G. Krein. Introduction to the theory of linear non-self-adjoint operators, Transl. Math. Monographs, vol. 18, Amer. Math. Soc., Providence, R.I., 1969.

\bibitem{dirac1}
Ya.~V.~Mykytyuk and D.~V.~Puyda, {\it Inverse spectral problems for Dirac operators on a finite interval}, J. Math. Anal. Appl. {\bf 386} (2012), 177--194.

\bibitem{sturm}
Ya.~V.~Mykytyuk and N.~S.~Trush, {\it Inverse spectral problems for {S}turm{--}{L}iouville operators with matrix-valued potentials}, Inverse Problems {\bf 26} (2010), no.~015009.

\bibitem{trushzeros}
N.~Trush, {\it Asymptotics of singular values of entire matrix-valued sine-type functions}, Mat. Stud. {\bf 30} (2008), 95--97.

\bibitem{kato}
T~Kato. Perturbation theory for linear operators. -- Berlin-Heidelberg-New York: Springer-Verlag, 1966.

\end{thebibliography}
\end{document}